\documentclass[a4paper,10pt]{amsart}

\usepackage{amsmath}
\usepackage{amssymb}
\usepackage{amsthm}
\usepackage{verbatim}
\usepackage[all]{xy}
\usepackage{enumerate}
\usepackage{dsfont}
\usepackage{color}

\newtheorem{thm}{Theorem}[section]

\newtheorem{cor}[thm]{Corollary}

\theoremstyle{definition}

\newtheorem{dfn}[thm]{Definition}

\newtheorem{rmk}[thm]{Remark}

\numberwithin{equation}{section}

\numberwithin{equation}{section}
\newcommand{\cK}{\mathcal{K}}

\newcommand{\cL}{\mathcal{L}}
\newcommand{\cR}{\mathcal{R}}

\newcommand{\cB}{\mathcal{B}}

\newcommand{\minotimes}{\otimes_{{\rm min}}}

\newcommand{\op}{{\rm op}}

\newcommand{\RH}{H_{\mathbb{R}}}

\title[On the isomorphism class of $q$-Gaussian W$^\ast$-algebras for infinite variables]{On the isomorphism class of $q$-Gaussian W$^\ast$-algebras for infinite variables}

\date{\noindent \today.  
  MC is supported by the NWO Vidi grant VI.Vidi.192.018 `Non-commutative harmonic analysis and rigidity of operator algebras'.
 }

\author[Martijn Caspers]{Martijn Caspers}

\address{Martijn Caspers, TU Delft, EWI/DIAM,
	P.O.Box 5031,
	2600 GA Delft,
	The Netherlands}

\email{M.P.T.Caspers@tudelft.nl}

\begin{document}

\maketitle

\begin{abstract}
Let $M_q(\RH)$ be the $q$-Gaussian von Neumann algebra associated with a separable infinite dimensional real Hilbert space $\RH$  where $-1 < q < 1$. We show that $M_q(\RH) \not \simeq M_0(\RH)$ for $-1 < q \not = 0 < 1$. The C$^\ast$-algebraic counterpart of this result was obtained recently in \cite{BCKW}.  Using ideas of Ozawa we show that this non-isomorphism result also holds on the level of von Neumann algebras.
\end{abstract}

\section{Introduction}
Von Neumann algebras of $q$-Gaussian variables originate from the work of  Bo\.{z}ejko and Speicher \cite{BS} (see also \cite{BKS}). To a real Hilbert space $\RH$ and a parameter $-1 < q < 1$ it associates a von Neumann algebra $M_q(\RH)$. At parameter $q=0$ this assignment $\RH \mapsto M_q(\RH)$ is known as Voiculescu's free Gaussian functor. The dependence of $q$ of these von Neumann algebras has been an intriguing and very difficult problem. A breakthrough result in this direction was obtained by Guionnet-Shlyakhtenko \cite{GS} who showed that for finite dimensional $\RH$ for a range of $q$ close to 0 all von Neumann algebras $M_q(\RH)$ are isomorphic. The range for which isomorphism is known decreases as the dimension $\RH$ becomes larger. The Guionnet-Shlyakhtenko approach is based on free analogues of (optimal) transport techniques. Their result also relies on existence and power series estimates of conjugate variables obtained by Dabrowski \cite{Dabrowski}. In fact the free transport techniques provide even an isomorphism result of underlying $q$-Gaussian C$^\ast$-algebras.

In case $\RH$ is infinite dimensional the isomorphism question of $q$-Gaussian algebras was addressed by Nelson and Zeng \cite{NelsonZeng}. They showed that for {\it mixed} $q$-Gaussians for which the array  $(q_{ij})_{ij}$ of commutation coefficients decays fast enough to 0 one obtains isomorphism of mixed $q$-Gaussian C$^\ast$- and von Neumann algebras. However, the isomorphism question for the original (non-mixed) $q$-Gaussians remained open, see Questions 1.1 and 1.2 of \cite{NelsonZeng}.  In \cite{BCKW} we showed that on the level of C$^\ast$-algebras there exists a non-isomorphsim result. In the current note we improve on this result: we show that for an infinite dimensional separable real Hilbert space $\RH$ and $-1 < q <1, q \not = 0$ we have $M_q(\RH) \not \simeq M_0(\RH)$.   This then fully answers Questions 1.1 and 1.2 of \cite{NelsonZeng} and provides a stark contrast to the results of Guionnet-Shlyakhtenko for  finite dimensional $\RH$.

The distinguishing property of $M_q(\RH)$ and $M_0(\RH)$ is a variation of the Akemann-Ostrand property that was suggested in a note by Ozawa \cite{OzawaNote} (see also \cite{PetersonEtAl}) and which we shall call W$^\ast$AO. We formally define it in Definition \ref{Eqn=WAO}. The most important novelty is that we quotient $\cB(L^2(M))$ by the C$^\ast$-algebra $\cK_M$ which is much larger than the ideal of compact operators on $L^2(M)$. This larger quotient turns out to provide von Neumann algebraic descriptions of the Akemann-Ostrand property \cite{OzawaNote}. We use this to distinguish  $M_q(\RH)$ and $M_0(\RH)$.

\vspace{0.3cm}

\noindent {\bf Acknowledgements:} The author thanks Mateusz Wasilewski, Changying Ding and Cyril Houdayer for comments on an earlier draft of this paper.

\section{Preliminaries}
 $\cB(X, Y)$ denotes the bounded operators between Banach spaces $X \rightarrow Y$.  $\mathcal{K}(X, Y)$  denotes the compact operators, meaning that they map the unit ball to a relatively compact set. We set $\cB(X) := \cB(X, X)$ and $\cK(X) := \cK(X, X)$.

The algebraic  tensor product (vector space tensor product) is denoted by $\otimes_{{\rm alg}}$   and $\minotimes$ is the minimal tensor product of C$^\ast$-algebras. $\otimes$ is used for tensor products of elements.

We refer to \cite{Takesaki1} as a standard reference on von Neumann algebras.  For a von Neumann algebra $M$ we denote by $(M, L^2(M), J, L^2(M)^+)$ its standard form. For $x \in M$ we write $x^{\op} := J x^\ast J$ which is the right multiplication with $x$ on the standard space.
  For a finite von Neumann algebra $M$ with trace $\tau$ we have $M \subseteq L^2(M)$ where $L^2(M)$ is the completion of $M$ with respect to the inner product $\langle x, y \rangle = \tau(y^\ast x)$. Therefore every $T \in \cB(L^2(M))$ determines a map $Q_0(T) \in \cB(M, L^2(M))$ given by $x \mapsto T(x)$. Set
\[
Q_1: \cB(L^2(M)) \rightarrow \cB(M, L^2(M)) \slash \cK(M, L^2(M)): T \mapsto Q_0(T) + \cK(M, L^2(M)).
\]
 $Q_1$ is clearly continuous and we define the closed left-ideal $\cK_{M}^L = \ker(Q_1)$ and the hereditary C$^\ast$-subalgebra $\cK_M = (\cK_{M}^L)^\ast \cap \cK_{M}^L$ of $\cB(L^2(M))$ (see also \cite{OzawaNote}). We let $\mathcal{M}(\cK_M) \subseteq \cB(L^2(M))$ be the multiplier algebra of $\cK_M$; indeed this multiplier algebra is faithfully   represented on $L^2(M)$  by \cite[Proposition 2.1]{Lance}. Then $\cK_M$ is an ideal in the C$^\ast$-algebra $\mathcal{M}(\cK_M)$. We have $M \subseteq  \mathcal{M}(\cK_M)$ and $M^\op \subseteq \mathcal{M}(\cK_M)$.

\subsection{A von Neumann version of the Akemann-Ostrand property}

\begin{dfn}
Let $M$ be a finite von Neumann algebra. We say that $M$ has W$^\ast$AO if the map
\begin{equation}\label{Eqn=WAO}
\theta: M \otimes_{{\rm alg}} M^{\op} \rightarrow \mathcal{M}(\cK_M)  \slash \cK_M: a \otimes b^{\op} \mapsto a b^{\op} + \cK_M.
\end{equation}
is continuous with respect to the minimal tensor norm and thus extends to a $\ast$-homomorphism $M \minotimes  M^{\op} \rightarrow  \mathcal{M}(\cK_M)  \slash \cK_M$.
\end{dfn}

We recall the following from \cite[Section 4]{OzawaNote}.
Let $\Gamma$ be a discrete group and let $\cL(\Gamma)$ and $\cR(\Gamma)$ be the left and right group von Neumann algebra respectively acting on $\ell^2(\Gamma)$. In this case $L^2(\cL(\Gamma)) \simeq \ell^2(\Gamma)$ as bimodules with the natural left and right actions of $\cL(\Gamma)$ and $\cR(\Gamma)$ on $\ell^2(\Gamma)$. We have $J \delta_s = \delta_{s^{-1}}$ which extends to an antilinear isometry on $\ell^2(\Gamma)$. Then $\cR(\Gamma) = J \cL(\Gamma) J$.

 Assume $\Gamma$ is icc so that $\cL(\Gamma)$ and $\cR(\Gamma)$ are factors, i.e. $\cL(\Gamma) \cap \cR(\Gamma) = \mathbb{C} 1$.   The map
\[
\pi: C^\ast( \cL(\Gamma),  \cR(\Gamma)     ) \rightarrow \cB(\ell^2(\Gamma) \otimes \ell^2(\Gamma)): a b^{\op} \mapsto a \otimes b^{\op}, \qquad a \in \cL(\Gamma), b^{\op} \in \cR(\Gamma).
\]
is a well-defined $\ast$-homomorphism by Takesaki's theorem on minimality of the spatial tensor product. In \cite[Section 4, Theorem]{OzawaNote} Ozawa showed the following theorem.

\begin{thm}\label{Thm=Ozawa}
Let $\Gamma$ be an exact icc group such that the $\ast$-homomorphism
\[
C_r^\ast(\Gamma) \otimes_{{\rm alg}} C_r^\ast(\Gamma)^{\op} \rightarrow \cB(\ell^2(\Gamma)) \slash \cK( \ell^2(\Gamma) ): a \otimes b^{\op} \mapsto a b^{\op} + \cK( \ell^2(\Gamma) ),
\]
is continuous with respect to $\minotimes$. Then $\cL(\Gamma)$ has W$^\ast$AO.
\end{thm}
\begin{proof}
By  \cite[Section 4, Theorem]{OzawaNote} we have
\[
\ker(\pi) = \mathcal{K}_{\cL(\Gamma)} \cap C^\ast( \cL(\Gamma),  \cR(\Gamma)     ).
\]
Therefore,
\[
\cL(\Gamma) \minotimes \cR(\Gamma) \rightarrow\!\!\!\!\!^{\simeq \pi^{-1}} C^\ast( \cL(\Gamma),  \cR(\Gamma)     )  \slash (\mathcal{K}_{\cL(\Gamma)} \cap C^\ast( \cL(\Gamma),  \cR(\Gamma)     )),
\]
which concludes the theorem.
\end{proof}

\begin{rmk}
It follows that if $\Gamma$ is an icc group that is bi-exact (or said to be in class $\mathcal{S}$, see \cite[Section 15]{BrownOzawa}) then $\cL(\Gamma)$ has W$^\ast$AO.
\end{rmk}

\subsection{$q$-Gaussians}
Let $-1 < q < 1$ and let $\RH$ be a real Hilbert space with complexification $H := \RH \oplus i \RH$. Set the symmetrization operator $P_q^k$ on $H^{\otimes k}$,
\begin{equation} \label{Eqn=Pq}
P_q^k(\xi_1 \otimes \ldots \otimes \xi_n) = \sum_{\sigma \in S_k} q^{i(\sigma)} \xi_{\sigma(1)} \otimes \ldots \otimes \xi_{\sigma(n)},
\end{equation}
where $S_k$ is the symmetric group of permutations of $k$ elements and $i(\sigma) := \# \{ (a,b) \mid a < b, \sigma(b) < \sigma(a) \}$ the number of inversions. The operator $P_q^k$ is positive and  invertible \cite{BS}. Define a new inner product on $H^{\otimes k}$ by
\[
\langle \xi, \eta \rangle_q := \langle P_q^k \xi, \eta \rangle,
\]
and call the new Hilbert space $H_q^{\otimes k}$. Set the Hilbert space direct sum $F_q(H) := \mathbb{C} \Omega \oplus (\oplus_{k=1}^\infty H_q^{\otimes k})$ where $\Omega$ is a unit vector called the vacuum vector. For $\xi \in H$ let
\[
l_q(\xi) ( \eta_1 \otimes \ldots \otimes \eta_k) := \xi \otimes \eta_1 \otimes \ldots \otimes \eta_k, \qquad l_q(\xi) \Omega = \xi,
\]
and then $l_q^\ast(\xi) = l_q(\xi)^\ast$. These `creation'  and `annihilation' operators are bounded and extend to $F_q(H)$. We define a  von Neumann algebra by the double commutant
\[
M_q(\RH) :=   \{ l_q(\xi) + l_q^\ast(\xi) \mid \xi \in \RH \}''.
\]
  Then $\tau_\Omega(x) := \langle x \Omega, \Omega \rangle$ is a faithful tracial state on $M_q(\RH)$ which is moreover normal. Now $F_q(H)$ is the standard form Hilbert space of $M_q(\RH)$ and $J x \Omega = x^\ast \Omega$.
For vectors $\xi_1, \ldots, \xi_k \in H$ there exists a unique operator $W_q(\xi_1 \otimes \ldots \otimes \xi_k) \in M_q(\RH)$ such that
\[
W_q(\xi_1 \otimes \ldots \otimes \xi_k) \Omega = \xi_1 \otimes \ldots \otimes \xi_k.
\]
These operators are called Wick operators. It follows that $W_q(\xi)^\op \Omega = \xi$.

\begin{rmk}\label{Remark=WAOZero}
Let $\mathbb{F}_\infty$ be the free group with countably infinitely many generators. $\mathbb{F}_\infty$ is icc and exact \cite{BrownOzawa} and hence Theorem \ref{Thm=Ozawa} applies. We conclude that $\cL(\mathbb{F}_\infty)$ has W$^\ast$AO. We have that $\cL(F_\infty) \simeq \Gamma_0(\RH)$ with $\RH$ a separable infinite dimensional real Hilbert space (see \cite[Theorem 2.6.2]{DykVoic}) and so  $\Gamma_0(\RH)$ has the W$^\ast$AO.
\end{rmk}

\section{Non-isomorphism of $q$-Gaussian von Neumann algebras}\label{Sect=qGaussian}

 The following theorem provides a necessary condition for W$^\ast$AO.

\begin{thm} \label{Thm=AntiAO}
Let $M$ be a finite von Neumann algebra with finite normal faithful tracial state $\tau$. Suppose there exists a unital von Neumann subalgebra $B \subseteq M$ and infinitely many subspaces $M_i \subseteq M, i \in \mathbb{N}$ that are left and right $B$-invariant and mutually $\tau$-orthogonal in the sense that $\tau(y^\ast x) = 0$ for $x \in M_i, y \in M_j, i\not = j$. Suppose moreover that there exists $\delta > 0$ and finitely many operators $b_j, c_j \in B$, with $ \sum_j b_j \otimes c_j^{\op}$ non-zero,  such that for every $i \in \mathbb{N}$ we have
\begin{equation}\label{Eqn=NonContraction}
\Vert  Q_0(\sum_{j} b_j c_j^{op}) \Vert_{\cB( M_i , L^2(M_i) )}  \geq (1 + \delta)  \Vert   \sum_j b_j \otimes c_j^{\op}  \Vert_{ B \minotimes B^{\op} }.
\end{equation}
Then $M$ does not have W$^\ast$AO.
\end{thm}
\begin{proof}
Let $X$ be the set of finite rank operators $x \in \cB(L^2(M))$ such that there exists $I_x \subseteq I$ finite with $\ker(x)^\perp \subseteq \oplus_{i \in I_x} L^2(M_i)$. Take $x \in X$   and choose $k \in I \backslash I_x$. Then,
\[
\begin{split}
 \Vert Q_0(\sum_{j} b_j c_j^{op} + x) \Vert_{\cB(M, L_2(M) )  }  \geq &
\Vert Q_0(\sum_{j} b_j c_j^{op} + x) \Vert_{\cB(M_k, L^2(M_k) )  } \\
= &
\Vert Q_0(\sum_{j} b_j c_j^{op}) \Vert_{\cB(M_k, L^2(M_k) )  } \\
\geq & (1 + \delta)  \Vert \sum_j b_j \otimes c_j^{\op} \Vert_{B \minotimes B^{\op}}.
\end{split}
\]
The operators in $X$ are norm dense in $\cK( L^2(M) )$ and by \cite[Section 2, Proposition]{OzawaNote} we have that $Q_0(\cK( L^2(M) ))$ is dense in $Q_0(\cK^L_M)$ in the norm of $\cB(M, L^2(M))$. As $Q_0$ is contractive $Q_0(X)$ is dense in $Q_0(\cK^L_M)$.  It therefore follows that for any $x \in \cK_M^L$ we have
\[
 \Vert Q_0(\sum_{j} b_j c_j^{op} + x) \Vert_{\cB(M, L_2(M) )  }  \geq (1 + \delta)  \Vert \sum_j b_j \otimes c_j^{\op} \Vert_{B \minotimes B^{\op}}.
\]
Since $Q_0$ is contractive for every $x \in \cK_M^L$ we have,
\[
 \Vert \sum_{j} b_j c_j^{op} + x \Vert_{\cB(L^2(M) )  }  \geq (1 + \delta)  \Vert \sum_j b_j \otimes c_j^{\op} \Vert_{B \minotimes B^{\op}}.
\]
Hence, certainly for the Banach space quotient norm we have
\[
 \Vert \sum_{j} b_j c_j^{op} + \cK_M  \Vert_{\cB(L^2(M) ) \slash \cK_M  }  \geq (1 + \delta)  \Vert \sum_j b_j \otimes c_j^{\op} \Vert_{B \minotimes B^{\op}}.
\]
As the left hand side norm is the norm of the C$^\ast$-quotient $\mathcal{M}(\cK_M) \slash \cK_M$ this concludes the proof (see \cite[Proposition 2.1]{Lance} as in the preliminaries).
\end{proof}

The proof of the following theorem essentially repeats its C$^\ast$-algebraic counterpart from \cite[Theorem 3.3]{BCKW}.

\begin{thm}\label{Thm=AntiAOGaussian}
Assume $\dim(\RH) = \infty$ and $-1 < q < 1, q \not = 0$. Then the von Neumann algebra $M_q(\RH)$ does not have W$^\ast$AO.
\end{thm}
\begin{proof}
Let $d \geq 2$ be such that $q^2 d > 1$.  Let
\[
M := M_q(\mathbb{R}^d \oplus \RH), \qquad  B := M_q(\mathbb{R}^d \oplus 0).
\]
Let $\{ f_i \}_i$ be an infinite set of orthogonal vectors in $0 \oplus \RH$  such that $\Vert W_q(f_i) \Vert = 1$. Let $M_{q,i} :=  B W_q(f_i) B$ which is a $B$-$B$ invariant subset of $M$. Then $M_{q,i}$ and $M_{q,j}$ are $\tau_\Omega$-orthogonal if $i \not = j$. For $k \in \mathbb{N}$ let
\[
\cB(k)  = \{ W_q(\xi) \mid \xi \in (\mathbb{R}^d \oplus 0)^{\otimes k} \}.
\]
It is proved in \cite[Equation (3.2)]{BCKW} that for $b,c \in \cB(k)$ we have
\[
 \langle b W_q(f_i) c \Omega, f_i \rangle_q =    \langle b c^{\op} f_i , f_i \rangle_q = q^k \langle b c^{\op} \Omega , \Omega \rangle_q.
\]
 Then for finitely many $b_j, c_j \in \cB(k)$ we have
\begin{equation}\label{Eqn=LowerEstimate}
\Vert Q_0(\sum_j b_j c_j^{\op}) \Vert_{\cB(M_{q,i}, L^2(M_{q,i}))} \geq
\Vert \sum_j b_j W_q(f_i) c_j \Vert_{ L^2(M_{q,i})} \geq
\vert \langle \sum_j b_j W_q(f_i) c_j \Omega, f_i \rangle_q \vert =
\vert \sum_j q^k \langle b_j \Omega c_j, \Omega \rangle_q \vert.
\end{equation}

Now let $\{ e_1, \ldots, e_d \} $ be an orthonormal basis of $\mathbb{R}^d \oplus 0$ and for $j = (j_1, \ldots, j_k) \in \{ 1,\ldots, d \}^k$ let $e_j = e_{j_1} \otimes \ldots \otimes e_{j_k}$.  Let $J_k$ be the set of all such multi-indices of length $k$. So $\# J_k = d^ k$.  Set $\xi_j = (P^k_q)^{- \frac{1}{2} } e_j$ so that $\langle \xi_j, \xi_j \rangle_q = \langle P_q^k \xi_j, \xi_j \rangle = 1$.

Now \eqref{Eqn=LowerEstimate} yields that for all $k \geq 1$ and all $i$,
\[
\begin{split}
\Vert Q_0(\sum_{j \in J_k}  W_q(\xi_j)^\ast    W_q(\xi_j)^{\op}) \Vert_{\cB(M_{q,i} ,  L^2(M_{q,i}) )} \geq &  \sum_{j \in J_k} q^k \langle W_q(\xi_j)^\ast \Omega W_q(\xi_j), \Omega \rangle_q =
\sum_{j \in J_k}  q^k \langle \Omega W_q(\xi_j),  W_q(\xi_j) \Omega \rangle_q\\
 = & \sum_{j \in J_k}  q^k  \langle \xi_j,  \xi_j \rangle_q =    q^k d^{k}.
\end{split}
\]
From \cite[Proof of Theorem 2]{Nou} (or see \cite[Proof of Theorem 3.3]{BCKW}) we find,
\[
\Vert \sum_{j \in J_k} W_q(\xi_j)^\ast \otimes  W_q(\xi_j)^{\op} \Vert_{B \minotimes B^{\op}} \leq ( \prod_{i=1}^\infty (1 - q^i)^{-1}  )^3 (k+1)^2  d^{k/2}.
\]
Therefore, as  $q^2 d  > 1$  there exists $\delta > 0$ such that for  $k$ large enough we have for every $i$,
\[
\Vert Q_0(\sum_{j \in J_k} W_q(\xi_j)^\ast    W_q(\xi_j)^{\op}) \Vert_{\cB(M_{q,i}, L^2(M_{q,i})   ) }  \geq   (1+ \delta)    \Vert \sum_{j \in J_k} W_q(\xi_j)^\ast \otimes  W_q(\xi_j)^{\op} \Vert_{B \minotimes B^{\op}}.
\]
Hence the assumptions of Theorem \ref{Thm=AntiAO} are witnessed which shows that W$^\ast$AO does not hold.
\end{proof}

\begin{cor}
Let $\RH$ be an infinite dimensional real separable Hilbert space. The von Neumann algebras $\Gamma_0(\RH)$ and $\Gamma_q(\RH)$ with $-1 < q < 1, q \not = 0$ are non-isomorphic.
\end{cor}
\begin{proof}
This is a consequence of Theorem \ref{Thm=AntiAOGaussian} and Remark \ref{Remark=WAOZero} as  W$^\ast$AO is preserved under isomorphism.
\end{proof}


\begin{thebibliography}{9}


\bibitem[BCKW22]{BCKW}
   M. Borst, M. Caspers, M. Klisse, M. Wasilewski,
  \emph{On the isomorphism class of $q$-Gaussian C$^\ast$-algebras for infinite variables},
  to appear in Proceedings of the American Mathematical Society, arXiv: 2202.13640.


\bibitem[BoSp91]{BS}
  M. Bo\.{z}ejko, R. Speicher,
  \emph{An example of a generalized Brownian motion},
   Comm. Math. Phys. {\bf 137} (1991), no. 3, 519--531.

\bibitem[BKS97]{BKS}
   M. Bo\.{z}ejko, B. K\"ummerer, R. Speicher,
   \emph{$q$-Gaussian processes: non-commutative and classical aspects},
    Comm. Math. Phys. {\bf 185} (1997), no. 1, 129--154.




\bibitem[BrOz08]{BrownOzawa}
   N. Brown, N. Ozawa,
  \emph{C$^\ast$-algebras and finite-dimensional approximations},
    Graduate Studies in Mathematics, 88. American Mathematical Society, Providence, RI, 2008. xvi+509 pp.




\bibitem[Dab14]{Dabrowski}
   Y. Dabrowski,
   \emph{A free stochastic partial differential equation},
   Ann. Inst. Henri Poincaré Probab. Stat. {\bf 50} (2014), no. 4, 1404--1455.


\bibitem[DKP22]{PetersonEtAl}
   C. Ding, S. Kunnawalkam Elayavalli, J. Peterson,
   \emph{Properly Proximal von Neumann Algebras},
  arXiv:2204.00517.


\bibitem[DNV92]{DykVoic}
   K. Dykema, A. Nica, D. Voiculescu,
   \emph{Free Random Variables},
       in the series CRM Monograph Series (volume 1), American Mathematical Society, Providence (1992).



\bibitem[GuSh14]{GS}
   A. Guionnet, D. Shlyakhtenko,
   \emph{Free monotone transport},
   Invent. Math. {\bf 197} (2014), no. 3, 613--661.



\bibitem[Lan95]{Lance}
  E. Lance,
  \emph{Hilbert C$^\ast$-modules,
   A toolkit for operator algebraists.}
    London Mathematical Society Lecture Note Series, 210. Cambridge University Press, Cambridge, 1995. x+130 pp.


\bibitem[NeZe18]{NelsonZeng}
   B. Nelson, Q. Zeng,
   \emph{Free monotone transport for infinite variables},
   Int. Math. Res. Not. IMRN 2018, no. 17, 5486--5535.



\bibitem[Nou04]{Nou}
    A. Nou,
    \emph{Non injectivity of the $q$-deformed von Neumann algebra},
    Math. Ann. {\bf 330} (2004), no. 1, 17--38.

\bibitem[Oza10]{OzawaNote}
   N. Ozawa,
  \emph{A comment on free group factors},
  Noncommutative harmonic analysis with applications to probability II, 241–-245, Banach Center Publ., 89, Polish Acad. Sci. Inst. Math., Warsaw, 2010.





\bibitem[Tak79]{Takesaki1}
  M. Takesaki,
  \emph{Theory of operator algebras. I.}
   Springer-Verlag, New York-Heidelberg, 1979. vii+415 pp.


\end{thebibliography}
\end{document}